\documentclass{amsart}

\usepackage{amsthm}
\usepackage{amsfonts}
\usepackage{amsmath}
\usepackage{enumerate}
\usepackage{amssymb}	
\usepackage{verbatim}
\usepackage{cite}
\usepackage{tikz-cd}
\usepackage{hyperref}

\usetikzlibrary{matrix,arrows,decorations.pathmorphing}

\DeclareMathAlphabet{\mathpzc}{OT1}{pzc}{m}{it}

\newtheorem{theorem}{Theorem}[section]
\newtheorem{proposition}[theorem]{Proposition}
\newtheorem{lemma}[theorem]{Lemma}
\newtheorem{corollary}[theorem]{Corollary}

\theoremstyle{definition}

\newtheorem{definition}{Definition}[section]

\theoremstyle{remark}
\newtheorem{remark}{Remark}
\newtheorem{example}{Example}[section]

\newcommand{\R}{\mathbb{R}}
\newcommand{\Z}{\mathbb{Z}}

\newcommand{\tth}{^{\mbox{\rm{\scriptsize{th}}}}}

\newcommand{\abs}[1]{\left| #1 \right|}
\newcommand{\of}{\circ}
\newcommand{\mbf}{\mathbf}

\newcommand{\norm}[1]{\abs{\abs{#1}}}
\newcommand{\set}[1]{\left\lbrace #1 \right\rbrace}

\newcommand{\mc}{\mathcal}
\newcommand{\mf}{\mathfrak}

\def\id{\operatorname{id}}
\def\Diff{\operatorname{Diff}}

\def\ad{\operatorname{ad}}

\def\Lie{\operatorname{Lie}}

\title[Cocycle Rigidity of New Partially Hyperbolic Actions]{Cocycle Rigidity of Partially Hyperbolic Abelian Actions with Almost Rank One Factors}
\author{Kurt Vinhage}

\begin{document}

\begin{abstract}
We extend the recent progress on the cocycle rigidity of partially hyperbolic homogeneous abelian actions to the setting with rank 1 factors in the universal cover. The method of proof relies on the periodic cycle functional and analysis of the cycle structure, but uses a new argument to give vanishing.
\end{abstract}

\maketitle

\footnotetext{This material is based upon work supported by the National Science Foundation under Award DMS-1604796}

\section{Introduction}

Homogeneous actions provide a useful testing ground and class of models for smooth dynamical systems. We study vector-valued cocycles over homogeneous higher-rank abelian actions which have certain hyperbolicity assumptions. Such systems were first studied by Katok and Spatizer \cite{ks94,ks97} in the case of {\it Anosov} systems, ones in which admit the maximal amount of hyperbolicity. In particular, local smooth rigidity and cocycle rigidity were established for such systems.

In the subsequent years, much progress was made in understanding systems which had less hyperbolicity, see, for instance \cite{dk2005,dk2011,dktoral,zwang-1,zwang-2,vinhage,vinhage-wang,ramirez-diss,damjanovic07,ks94-partial}. In the case of (suspensions of) automorphism actions on tori, a complete solution for was found, both for cocycle and local rigidity, through an adaptation of the classical KAM (Kolmogorov-Arnold-Moser) theory \cite{dktoral} utilizing Fourier analysis or representation theory to decompose subspaces.

Here we employ the other developed approach, the geometric method, which has a few advantages. The first is that the solution to the cocycle is built with explicitly defined functions which are easily seen to converge, instead of producing a formal solution which one later shows is a function. The second is the hope that a local rigidity result may come from the framework provided by our argument. To obtain local rigidity, H\"older cocycle rigidity for {\it perturbations} of the flows we consider is needed. The geometric method has historically proven more robust under perturbations than the representation theory approach.


The principal advantage is the ability to handle the H\"older category. In some cases, cocycle rigidity for smooth cocycles over the actions we consider can be obtained through the representation theory method, as in \cite{ks94-partial}, \cite{miecz}, \cite{ramirez-diss} and \cite{dk-parabolic}. That approach may also produce results in H\"older regularity, but coniderable additional input is needed. For smooth cocycles, regularity considerations are handled through Sobolev spaces, generation of the tangent bundle through stable and unstable subspaces, and elliptic regularity results, which are unavailable in the H\"older category. A new proof of regularity would have to be produced. Furthermore, certain decay of correlations results in the presence $\mf{so}(n,1)$ and $\mf{su}(n,1)$ factors  are needed to produce an argument similar to \cite{ks94-partial}, which likely hold but are not part of the standard literature.

\subsection*{Statement of Main Result and Proof Scheme}

We state our main theorem here, all terms appearing Theorem \ref{thm:main} are defined subsequently.

\begin{theorem}
\label{thm:main}
Let $\alpha$ be a generic restriction of a Weyl chamber flow or twisted Weyl chamber flow. Then if $\beta$ is a H\"older cocycle over $\alpha$ taking values in a vector space, $\beta$ is cohomologous to a constant cocycle via a H\"older continuous transfer function.
\end{theorem}

The proof scheme follows the geometric method: a potential for the cocycle $\beta$ called the {\it periodic cycle functional} (Section \ref{sec:pcf}) is introduced on families of contracting and expanding foliations for the dynamics called the {\it Lyapunov foliations} (Section \ref{sec:path-grp}). These functions are candidates for a transfer function, but a well-definedness problem occurs: there are many ways to define the transfer function depending on which path along the Lyapunov foliations one chooses. To prove well-definedness, one must verify that the functional vanishes on all cycles of Lyapunov folitations.

In the homogeneous case, much is known about how these foliations form a web throughout $G/\Gamma$ or $G^\pi / \Lambda$. In particular, the way in which the web is woven is directly related to a certain topological group extension of $G$. Our analysis will then revolve around analyzing features of the path and cycle structure, and lead us to the theory of central extensions. Unlike previous cases (the first of which was \cite{dk2011}) we use the {\it Hopf presentation} (Section \ref{sec:hopf}) of the universal central extension, rather than the {\it Matsumoto presentation}. It should be noted that ending up with the universal central extension of $G$ is very unexpected: our dynamical and algebraic arguments to arrive there are much different.

\begin{remark}
Theorem \ref{thm:main} hints at a local rigidity result: following \cite{dk2011}, one can reduce the local rigidity problem to cocycle rigidity over a smooth perturbation $\widetilde{\alpha}$ of $\alpha$. In \cite{dk2011}, however, there is a canonical way of associating key related spaces for the perturbation and homogeneous actions. Such associations are available here as well, but certain features do not follow due to the difference in the argument.
\end{remark}

\subsection*{Restrictions of Weyl Chamber Flows}

Let $G$ be a real, simply connected group, which is a direct product of simple noncompact groups $G = G_1 \times \dots \times G_n$, $n \ge 2$. Let $\mf g = \Lie(G)$, $\mf g_i = \Lie(G_i$), and $\Gamma \subset G$ a lattice projecting densely onto each factor and admitting no homomorphisms to $\R$ (for instance, any irreducible lattice \cite[II.6.7, IX.6.14, IX.6.20]{margulis91}). Then $G$ comes equipped with projections $\pi_i : G \to G_i$ onto the $i\tth$ factor. Choose fixed Cartan subalgebras $\mf a_i \subset \Lie(G_i)$, so that $\mf A = \bigoplus \mf a_i$ is a Cartan subalgebra of $\mf g$. Finally, let $\Delta_{G_i}$ be the roots of $G_i$, and $\Delta_G = \bigsqcup \Delta_{G_i}$ denote the roots of $G$.

A homomorphism $\iota : \R^k \times \Z^l \to \mf A$ is said to be {\it generic} if for every $1 \le i \not= j \le n$:

\begin{enumerate}
\item[($\mf G 1$)] there exists some $\chi \in \Delta_{G_i}$ such that $(d\pi_i \of \iota)^*(\chi) \not= 0$
\item[($\mf G 2$)] whenever $\chi_1 \in \Delta_{G_i},\chi_2 \in \Delta_{G_j}$ with $i \not= j$, then $(d\pi_i \of \iota)^*(\chi_1) \not\propto (d\pi_j \of \iota)^*(\chi_2)$
\end{enumerate} 

 A subgroup $\mf a \subset \mf A$ will be called {\it generic} if $\mf a = \iota(\R^k \times \Z^l)$ for some generic $\iota$. For any $\iota$, we let
 
\begin{equation}
\label{eq:homogeneous-def}
\begin{array}{rccl}
\alpha_{\iota} :  & (\R^k \times \Z^l)  \times G / \Gamma & \to & G / \Gamma \\
 & (a,g\Gamma) & \mapsto & \exp(\iota(a))g\Gamma
\end{array}
\end{equation} 

be the homogeneous action determined by $\iota$. If $\iota$ is an isomorphism, $\alpha_\iota$ will be called a {\it (partially hyperbolic) Weyl chamber flow}.\footnote{We emphasize partially hyperbolic since we do not form a double quotient by the centralizer of $\mf A$} If $\iota$ is generic, we will call the action $\alpha_\iota$ a {\it generic restriction} of a Weyl chamber flow. 

\begin{example}
Consider the case of $\mf{g}_1  = \dots = \mf g_n = \mf{sl}(2,\R)$. Then $\mf a_1 = \dots = \mf a_n  = \set{\begin{pmatrix}
t & 0 \\
0 & -t
\end{pmatrix} : t \in \R}$ is a split Cartan subalgebra. Then an element of $\mf A$ is determined by some $n$-tuple
$(t_1,\dots,t_n)$. The Weyl chamber flow is then the full action of $\R^n$.
A homomorphism $\iota : \Z^2_{(k,l)} \to \mf A$ is then determined by a sequence of vectors
$(x_i,y_i)$, and $\iota(k,l) = (x_1k + y_1l ,\dots,x_nk +y_nl)$. $\iota$ is then generic if and only if all vectors
$(x_i,y_i)$ are nonzero and pairwise nonproportional.
\end{example}

We note that our actions may have the following feature, which was omitted from other works:

\begin{definition}
\label{def:almost-rank-one}
Let $G$, $\Gamma$ and $\iota$ be as above. A homogeneous action $\alpha_\iota$ defined by \eqref{eq:homogeneous-def} has an {\it almost rank one factor} if there exists some $1 \le i \le n$ and fucntional $\bar{\chi} : \R^k \times \Z^l \to \R$ such that every $\chi \in \Delta_{G_i}$ satisfies $(d\pi_i \of \iota)^*\chi \propto \bar{\chi}$.
\end{definition}

Observe that genericity conditions $(\mf G 1)$ and $(\mf G 2)$ do not rule out almost rank one factors. Also note that since the roots $\Delta_{G_i}$ contain a dual basis for $\mf a_i$, an almost rank one factor is equivalent to $d\pi_i \of \iota(\R^k \times \Z^l)$ being contained in a 1-dimensional subalgebra, in particular we allow copies of $\mf{so}(n,1)$ and $\mf{su}(n,1)$. Furthermore, if the lattice $\Gamma$ was not irreducible and projected onto a lattice in an almost rank one factor, then this factor would be a homogeneous rank one factor which immediately implies nontrivial cohomology. H\"older cocycle rigidity for actions without almost rank one factors was obtained in \cite{vinhage-wang}.

\subsection*{Restrictions of Twisted Weyl Chamber Flows}

Let $G$ and $\Gamma$ be as above, and $\pi : G \to GL(V)$ be a representation of $G$ on a finite dimensional vector space $V$. Then the representation $\pi$ has weights with respect to $\mf A \subset \mf g$, denote them by $\Phi_\pi$. We say that $\pi$ is {\it generic with respect to $\iota$} if:

\begin{enumerate}
\item[($\mf G 3$)] if $\chi \in \Phi_\pi$, $\iota^*\chi \not= 0$
\item[($\mf G 4$)] if $\chi_1 \in \Phi_\pi$ and $\chi_2 \in \Delta_{G_j}$, $\iota^*\chi_1 \not\propto (d\pi_j \of \iota)^*\chi_2$
\item[($\mf G 5$)] if $\chi_1,\chi_2 \in \Phi_\pi$, $\iota^*\chi_1$ and $\iota^*\chi_2$ are not negatively proportional
\end{enumerate}

With such a representation, one may form the semidirect product of groups $G^\pi = G \ltimes_\pi V$. We assume that there exists some lattice $\Lambda$ of the form $\Lambda = \Gamma \ltimes_\pi D$, where $D$ is a lattice in $V$. This can be built by considering $G$ (or some factor of $G$) as an algebraic group, and using an algebraic representation $\pi$ of $G$ so that the restriction to the arithmetic lattice preserves integer points. Then one may similarly define a homogeneous action on $G^\pi  / \Lambda$ via formula \eqref{eq:homogeneous-def}. Such an action is called a {\it twisted Weyl chamber flow} if $\iota$ is an isomorphism onto $\mf A$, and a {\it generic restriction} if genericity conditions ($\mf G 1$) - ($\mf G5$) are all satsified.

We establish the following notation: in many occurances, a result will hold for both restrictions of Weyl chamber flows and twisted Weyl chamber flows. We let $G^{(\pi)}$ either $G$ or $G^\pi$, depending on the setting. We let $\mf g^{(\pi)}$ denote its Lie algebra.

\subsection*{Cocycles}

If $\alpha : \R^k \times \Z^l \to \Diff(M)$ is an abelian action, a {\it cocycle} over $\alpha$ taking values in an abelian group $H$ is a map $\beta : (\R^k \times \Z^l) \times M \to H$ such that

\begin{equation}
\label{eq:cocycle-eq}
\beta(a+b,x) = \beta(a,\alpha(b)x) + \beta(b,x)
\end{equation}

Cocycles are useful as they can be used to quantify statistics or behaviour accumulating over the dynamics $\alpha$. The simplest cocycles are {\it constant} cocycles: ones which do not depend on the $x$-coordinate. In this case, the cocycle equation becomes the equation defining a homomorphism. A cocycle $\beta$ is said to be {\it cohomologous to a constant cocycle} if there exists a function $h : M \to H$ and a homomorphism $j : \R^k \times \Z^l \to H$ such that:

\[\beta(a,x) = h(\alpha(a)x) - h(x) + j(a) \]

It is easy to verify \eqref{eq:cocycle-eq} for any choice of $h$ and $j$. In this case, $h$ is called the {\it transfer function}. Note that we have intentionally omitted the regularity of $\beta$ and $h$, as in different cases, different regularities arise and are desired.

\subsection*{Acknowledgements}

The author would like to thank Anatole Katok and Zhenqi Wang for many useful comments improving the exposition of this paper.

\section{Normally Hyperbolicity and the Path and Cycle Groups}
\label{sec:path-grp}

In this section, we recall some key constructions around the periodic cycle functional and cycle structures for partially hyperbolic actions. An action $\alpha : \R^k \times \Z^l \to \Diff^\infty(M)$ acts {\it uniformly normally hyperbolically} with respect to a foliation $\mc N$ with smooth leaves if for some $a \in \R^k \times \Z^l$, there is a splitting:

\[ TM = T\mc N \oplus E^s_a \oplus E^u_a \]

and \[\norm{d\alpha(a)|_{E^s_a}} < \lambda < m\left(d\alpha(a)|_{T\mc N}\right) \le \norm{d\alpha(a)|_{T\mc N}} <\lambda^{-1} < m\left(d\alpha(a)|_{E^u_a}\right) \]

In the case of a homogeneous action, these bundles exist for almost every $a \in A$, and can be described explicitly. Assume $\alpha_\iota$ is a homogeneous action defined as in \eqref{eq:homogeneous-def}. Then there exists a subset $\Delta \subset (\R^k \times\Z^l)^* \setminus \set{0}$, the Lyapunov exponents of the actions, such that the adjoint action of $\iota(\R^k \times \Z^l)$ has generalized eigenspace decomposition which write $\mf g^{(\pi)}$ as $\displaystyle\mf n \oplus \bigoplus_{\chi \in \Delta} \mf u_\chi$ (note that we do not assume that each $\ad_X$ is diagonazable, each $\mf u_\chi$ corresponds to a Jordan block). With this decomposition, the eigenvalue of $\ad_{\iota(X)}$ on $\mf u_\chi$ has real part $\chi(X)$ and $\mf n$ corresponds to the purely imaginary eigenvalues. In the case of the Weyl chamber flow, $\mf n = \mf A$ and the set of Lyapunov exponents is exactly the roots and weights, so we may employ genericity conditions ($\mf G 1$) - ($\mf G 5$). Note that condition ($\mf G3$) implies that $\mf n \subset \mf g$.

Each $U_\chi = \exp(\mf u_\chi)$ is a nilpotent subgroup of $G$ called a Lyapunov subgroup, and the coset foliation $\mc U_\chi = \set{ U_\chi g : g \in G}$ is the {\it Lyapunov foliation}.

\begin{remark}
More generally, a {\it coarse} Lyapunov foliation is $U_{(\chi)} = \prod_{\chi' \propto \chi} U_\chi$. One may, in a similar way, define the coarse Lyapunov foliation as the foliation tangent to the intersection of the form $\bigcap_i E^s_{a_i}$ having the smallest possible positive dimension. In this way, one can obtain similar foliations for non-homogeneous actions.
\end{remark}

\begin{lemma}
\label{lem:generates}
Conditions ($\mf G 1$) and ($\mf G 3$) imply that the $U_\chi$ (locally) generate $G^{(\pi)}$
\end{lemma}

\begin{proof}
Let us address the semisimple case first. Consider the group generated by the $U_\chi$. Since each $\chi$ is a restriction of a root on $\mf A$, each $U_\chi$ is a product of root subspaces $U_\chi = \prod_{r \in [\chi]} U_r$, where $[\chi] = \set{ r \in \Delta_G : \iota^*r = \chi}$ is the collection of roots which restrict to $\chi$. If $\iota^*r = \chi$, then $-\iota^*r = -\chi$, so $-\chi$ is a Lyapunov exponent, and $-r \in [-\chi]$. Then we may generate $\exp([\mf u_{r},\mf u_{-r}]) \subset \mf n$ using $U_\chi$ and $U_{-\chi}$ by standard Lie theory.
Note that if $\chi_1 \not= -\chi_2$, then $[\mf u_{\chi_1},\mf u_{\chi_2}] \subset \mf u_{\chi_1+\chi_2}$ (with the convention that $\mf u_{\chi_1+\chi_2} = 0$ if $\chi_1+\chi_2$ is not a Lyapunov exponent). In conclusion, since $[\mf n,\mf u_\chi] \subset \mf u_\chi$, we conclude that the group generated by the $U_\chi$ is a Lie subgroup (whose Lie algebra is generated by $ \left( \bigoplus_\chi [\mf u_{\chi}, \mf u_{-\chi}] \right) \oplus \left( \bigoplus_\chi \mf u_\chi \right)$.

Let $\mf g'$ denote the Lie algebra of the group generated by the $U_\chi$. Then again, by the Baker-Campbell-Hausdorff formula, we can express any element of $\mf g'$ as the exponential of some bracket polynomial in the $\mf u_\chi$. Note that if $X \in \mf u_\chi$ for some $\chi$, then applying $\ad_X$ to a bracket polynomial will keep it a bracket polynomial. If we take $Y \in \mf n$, then $\ad_Y(\mf u_\chi) \subset \mf u_\chi$, so it will continue to be a bracket polynomial. In particular, we get that the subalgebra is actually an ideal. Since it intersects every simple factor by ($\mf G 1$), we conclude that $\mf g' = \mf g$, and hence that the $U_\chi$ generate $G$.

In the twisted case, our arguments above show that the semisimple part is generated. Then ($\mf G 3$) immediately implies that all weight spaces are generated, so $G^\pi$ must be generated.
\end{proof}

\subsection{Lyapunov Paths and Cycles}

Let $\alpha : \R^k \times \Z^l \to \Diff^\infty(M)$ be an action acting normally hyperbolic with respect to $\mc N$ with Lyapunov foliations $\mc F_\chi$. Then a Lyapunov path is a sequence of points $\rho = (x_0,x_1,\dots,x_n)$ such that $x_{i+1} \in \mc F_{\chi_i}(x_i)$, $0 \le i \le n-1$. The corresponding sequence $(\chi_0,\dots,\chi_{n-1})$ is called the {\it combinatorial pattern} of $\rho$. $\rho$ is a {\it cycle} if $x_n = x_0$.

Let $\mf P$ denote the space of Lyapunov paths (with arbitrary base point). In the case of a homogeneous action, we introduce a topology on $\mf P$ by identifying the the paths with combinatorial pattern $(\chi_0,\dots,\chi_{n-1})$ as a subset of $\widetilde{M}^{n+1}$ (where $\widetilde{M}$ is the universal cover of $M$), then taking a disjoint union of all combinatorial patterns. Note that the space of all cycles with a fixed basepoint $x_0$ forms a group under concatenation, if we make the following identifications:

\begin{eqnarray}
\label{eq:out-n-back} (x_0,\dots,x_i,y,x_i,\dots,x_n) & \sim & (x_0,\dots,x_i,\dots,x_n) \\
\label{eq:redundant} (x_0,\dots,x_{i-1},x_i,x_i,x_{i+1},\dots,x_n) & \sim & (x_0,\dots,x_{i-1},x_i,x_{i+1},\dots,x_n)
\end{eqnarray}

\subsection{The path group, its subgroups and local sections}

If $\alpha$ is a homogeneous flow, let be $\mc P$ the space of Lyapunov paths with a fixed basepoint at $e \in G^{(\pi)}$. Then $\mc P$ (with identifications \eqref{eq:out-n-back} and \eqref{eq:redundant}) carries a group structure corresponding to the free product of the Lyapunov subgroups, $\mc P \cong \prod^*_{\chi \in \Delta} U_\chi$. We call $\mc P$ the {\it path group}, and its group structure is only guaranteed in the homogeneous setting. In geometric terms, if $\rho = (e,g_1,g_2,\dots,g_n)$ is a path based at $e$ let $\omega(\rho) = g_n$ be its endpoint and $\rho^x = (x,g_1x,\dots,g_nx)$ be the path right-translated by $x$. We then define $\rho_1 \cdot \rho_2 = \rho_1 * \rho_2^{\omega(\rho_1)}$. The following appeared as Proposition 4.2 and Lemma 6.2 in \cite{vinhage}:

\begin{proposition}
The identification topology on $\mc P$ is a locally path-connected, path-connected group topology. With the subspace topology, $\mc C$ and $\mc C^c$ are also locally path-connected.
\end{proposition}

Observe that $\omega : \mc P \to G^{(\pi)}$ is the canonical projection induced by the inclusions of the $U_\chi$ into $G^{(\pi)}$ guaranteed by the universal property of free products. We let $\mc C^c$ denote the group of contractible cycles (ie, cycles in $G$), so that $\mc C^c$ is a closed normal subgroup of $\mc P$ (since it is exactly the kernel of $\omega$). Similarly, we let $\mc C$ denote the space of cycles in $M$. This is a subgroup, but not normal, since it is $\omega^{-1}(\Gamma)$ (respectively, $\omega^{-1}(\Lambda)$), and $\Gamma$ (respectively, $\Lambda$) is not normal in $G$ (respectively, $G^\pi$). Recall conditions ($\mf G 1$) - ($\mf G 5$) for a generic action, and that we assume $G = G_1 \times \dots \times G_n$. Condition $(\mf G 2)$ guarantees that each $U_\chi \subset G_i$ for each $\chi \in \Delta_{G_i}$. Furthermore, condition ($\mf G 4$) shows that if $\chi \in \Phi_\pi$, $U_\chi \subset V$. We then let $\mc P_i$ denote the subgroup of $\mc P$ generated by $\set{U_\chi : \chi \in \Delta_{G_i}}$ and $\mc P_{ss}$ be the group generated by the $\mc P_i$ (ie, the semisimple components. Finally, let $\mc P_V$ denote the subgroup of $\mc P$ generated by $\set{U_\chi : \chi \in \Phi_\pi}$. The following Lemma also appeared in \cite[Lemma 5.7]{vinhage}, and the proof is identical:

\begin{lemma}
\label{lem:loc-sec}
$\omega$ has a local section $s$. Furthermore, $\omega|_{\mc P_i} : \mc P_i \to G_i$ and $\omega|_{\mc P_V} : \mc P_V \to V$ have local sections.
\end{lemma}

\subsection{Sufficient Subclasses}
\label{sec:suff}

We introduce a new notion of a {\it sufficient subclass}:

\begin{definition}
If $\mf P$ is the space of Lyapunov paths on $M$ (with arbitrary base point), a subset $\Omega \subset \mf P$ is said to be a {\it sufficient subclass of Lyapunov paths} if:

\begin{enumerate}[(i)]
\item whenever $\rho_1,\rho_2 \in \Omega$ are concatenable, $\rho_1 * \rho_2 \in \Omega$
\item whenever $\rho \in \Omega$, then $\bar{\rho}$, the reverse of $\rho$, is in $\Omega$
\item $\Omega$ is closed in $\mf P$ with the identification topology
\item If $x,y \in M$, there exists some $\rho \in \Omega$ which begins at $x$ and ends at $y$
\end{enumerate}
\end{definition}

Sufficient subclasses are meant to represent a subclass of paths which is still transitive in $M$ and carries a groupoid structure. In this way, a functional on $\Omega$ will induce a function on $M$ if and only if it vanishes on the cycles in $\Omega$. In the case of homogeneous actions, we may produce subclasses readily:

\begin{lemma}
\label{lem:subgroup-sufficient}
If $\mc X \subset \mc P$ is a closed subgroup of $\mc P$ such that $\omega(\mc X) = G^{(\pi)}$, then $\Omega = \set{ \rho^x : \rho \in \mc X, x \in G}$ is a sufficient subclass of Lyapunov paths.
\end{lemma}

\begin{proof}
We verify properties (i)-(iv). (i) and (ii) follow immediately from the fact that $X$ is a subgroup and the definition of group multiplication in $\mc P$. (iii) follows because $\mc X$ is closed in $\mc P$, and the fact that the coordinates for the cosets $U_\chi g$ vary continuously with $g$. (iv) follows from the fact that $\omega(X) = G$.
\end{proof}

Lemma \ref{lem:subgroup-sufficient} will be applied to the subgroup $[\mc P, \mc P]$, which remains transitive on any semisimple group. We define another subset of the path and cycle group. A cycle $\sigma$ will be called {\it stable} if there exists some $a \in \R^k \times \Z^l$ such that for every $\chi$ in the combinatorial pattern of $\sigma$, $\chi(a) < 0$. In other words, the path $\sigma$ is completely contained in the stable leaf of $\alpha(a)$. Then let $\mc S \subset \mc P$ be constructed in the following way: first, take the group generated by all stable cycles. Then, take its normal closure (the group generated all possible conjugations of elements of $\mc S$ by elements of $\mc P$). Finally, take the topolgoical closure in $\mc P$.

\begin{lemma}
\label{lem:path-directprod}
If $\iota$ is generic, then:

\begin{enumerate}[(a)]
\item for a generic restriction of a Weyl chamber flow, $\mc P / \mc S = \prod_{i=1}^n \mc P_i / \mc S_i$ and $[\mc P, \mc P]/ (\mc S \cap [\mc P, \mc P])$
\item for a generic restriction of a twisted Weyl chamber flow, $\mc P / \mc S = V \times \prod_{i=1}^n \mc P_i / \mc S_i$, and $([\mc P_{ss}, \mc P_{ss}] \cdot \mc P_V) / ( \mc S \cap ([\mc P_{ss}, \mc P_{ss}]\cdot \mc P_V) ) = V \times \prod_{i=1}^n [\mc P_i/ \mc S_i,\mc P_i / \mc S_i]$
\end{enumerate}
\end{lemma}

\begin{proof}
We first treat the case of Weyl chamber flows. The lemma follows once we can show that $[\rho_1,\rho_2] \in \mc S$, once $\rho_1 \in U_{\chi_1}$ and $\rho_2 \in U_{\chi_2}$ with $U_{\chi_2} \subset G_i \not= G_j \supset U_{\chi_2}$. Observe that condition ($\mf G 2$) implies that there exists an $X \in \R^k \times \Z^l$ such that $\chi_1(X), \chi_2(X) < 0$. Hence the cycle is stable, and the result follows for the semisimple case.

In the twisted case, observe that condition ($\mf G 5$) allows us to commute pairs within $\mc P_V$ modulo $\mc S$ and condition ($\mf G 4$) allows us to commute elements of $\mc P_V$ with elements of $\mc P_{ss}$. So since $V$ is abelian and generated by the subgroups making up $\mc P_V$, we get the result. 
\end{proof}

\section{The Periodic Cycle Functional}
\label{sec:pcf}

In this section we develop the main tool for showing cocycle rigidity, the periodic cycle functional. We recall the definition of the periodic cycle functional, and list some key properties, which can be found, for instance, in \cite[Section 4.4.3]{katok-nitica}, \cite{knt00} or \cite{dk2005}. If $\beta$ is a $\theta$-H\:older cocycle (taking values in $\R$) over a partially hyperbolic $\R^k \times \Z^l$ action  $\alpha$ with Lyapunov foliations $\mc U_\chi$, define:

\[ p_\beta(x,y) = \lim_{n \to \infty} \beta(a^n,x) - \beta(a^n,y) \]

where $x$ and $y$ are in the same leaf of $\mc U_\chi$ and $a$ is chosen in such that $\chi(a)$ is negative, $p_\beta$ converges, is independent of the choice of contracting $a$. Then we may use the functionals $p_\beta$ to induce a groupoid morphism $P_\beta : \mf P \to \R$ by defining:

\[ P_\beta(x_0,x_1,\dots,x_n) = \sum_{i=0}^{n-1} p_\beta(x_{i+1},x_i) \]

Clearly, $P_\beta(\rho_1 * \rho_2) = P_\beta(\rho_1) + P_\beta(\rho_2)$, but in most cases $P_\beta(\rho_1 \cdot \rho_2) \not= P_\beta(\rho_1) + P_\beta(\rho_2)$, so $P_\beta$ is not a homomorphism from $\mc P$. However, since the group structure restricted to $\mc C$ coincides with concatenation, $P_\beta$ can be considered a continuous homomorphism from $\mc C$.  Recall the subgroup $\mc S$ defined in Section \ref{sec:suff}. The following follows immediately from the definition of $P_\beta$ as a telescoping sum and the ability to choose a fixed contracting element $a$ for stable cycles:

\begin{lemma}
\label{lem:stable-vanishing}
$P_\beta(\mc S) = 0$
\end{lemma}

If $\Omega$ is a sufficient subclass of Lypaunov paths, let $\mc C(\Omega) = \mc C \cap \Omega$, the Lyapunov cycles in $\Omega$ at a fixed base point $x_0$.

\begin{proposition}
\label{prop:pcf-suff}
Let $\Omega$ be a sufficient subclass of Lyapunov paths for a partially hyperbolic action $\alpha$ and $x_0$ be a fixed base point. If $\beta$ is a cocycle over  $\alpha$ taking values in a vector space, then $\beta$ is cohomologous to a constant cocycle if and only if $P_\beta$ vanishes on $\mc C(\Omega)$.
\end{proposition}

\begin{proof}
Direct computation shows that if $\beta$ is cohomologous to a constant, $P_\beta$ vanishes on $\mc C$. So we wish to show the opposite. Define a function $h(x) = P_\beta(\rho_x)$, where $\rho_x$ is any path from $x_0$ ending at $x \in M$ and lying in $\Omega$. Then the vanishing of $P_\beta$ on $\mc C(\Omega)$ implies that $h$ is well-defined. Choose $x \in M$ and $a \in \R^k \times Z^l$. Then choose a path $\rho_x$ from $x_0$ to $x$ and a path $\rho_a$ from $x_0$ to $\alpha(a)x_0$. Then $\alpha(a)\rho_x$ is a path from $\alpha(a)x_0$ to $\alpha(a)x$ and:

\begin{eqnarray*}
 h(\alpha(a)x) & = & P_\beta(\rho_a * \alpha(a)\rho_x) \\
 & = & P_\beta(\rho_a) + P_\beta(\alpha(a)\rho_x) \\
 & = & P_\beta(\rho_a) + \sum_{i=0}^{n-1} p_\beta(\alpha(a)x_{i+1},\alpha(a)x_i) \\
 & = & P_\beta(\rho_a) + \sum_{i=0}^{n-1} \lim_{k \to \infty} \beta(a^k,\alpha(a)x_{i+1}) - \beta(a^k,\alpha(a)x_i) \\
 & = & P_\beta(\rho_a) + \sum_{i=0}^{n-1} \beta(a,x_i) - \beta(a,x_{i+1}) + \lim_{k \to \infty} \beta(a^{k+1},x_{i+1}) - \beta(a^{k+1},x_i)  \\
 & = & P_\beta(\rho_a) + \beta(a,x) - \beta(a,x_0) + h(x) \\
\beta(a,x) & = & h(\alpha(a)x) -h(x) + (\beta(a,x_0) - h(\alpha(a)x_0))
\end{eqnarray*}

The cocycle equation then guarantees that $j(a) = \beta(a,x_0) - h(\alpha(a)x_0)$ is a homomorphism. That $h$ is H\"older follows from H\"older transitivity of Lyapunov paths for homogeneous actions. 
\end{proof}

\section{Vanishing of Homomorphisms}

With the exception of Section \ref{sec:twisted-vanishing}, we deal only with Weyl chamber flows, not twisted Weyl chamber flows in this section. It contains the main new technique for proving vanishing. We follow the general strategy of \cite{vinhage} and \cite{vinhage-wang} by looking for obstructions to homomorphisms in the theory of topological groups and central extensions. However, we arrive at the central extension in a surprising new way: indeed, the contractible cycles $\mc C^c/ \mc S$ {\it will} have homomorphisms into abelian groups. In the case of $\Z^2 \curvearrowright (SL(2,\R) \times SL(2,\R) ) /\Gamma$, the negative resonance of the roots will mean that for a generic restriction $\mc P/ \mc S = (\R * \R)^2$. This group and its contractible cycle group clearly has homomorphisms to $\R$ (and so will its cycle group). Rather than looking at the contractible cycles separately, we will treat them together with the non-contractible ones. This will allow us to arrive at another classical presentation of the universal central extension of a group. 

\subsection{The Hopf Presentation}
\label{sec:hopf}

Let $H$ be any group (which may or may not carry a topology). A universal central extension of $H$ is a group $\mbf{H}$ with a projection $p : \mbf{H} \to H$ such that if $1 \to Z \to L \to H \to 1$ is any central extension of $H$, then there is a unique map $\phi : \mbf{H} \to L$ extending the identity on $H$.

Let $X$ be a set of generators for $H$. Note that $X$ may have arbitrary cardinality. Then let $\mc F(X)$ be the free group with generators $X$, and $\mc R(X)$ be the kernel of the obvious map $\mc F(X) \to H$ (ie, the relations on $H$ with respect to generators $X$). The following is the classical description for the universal central extension of a perfect group in a general algebraic setup. It was first shown by H. Hopf, and a modern treatment can be found in, for instance, \cite[Theorem 4.1.3]{rosenberg}:

\begin{proposition}[The Hopf Presentation]
If $H$ is a perfect group, then $\mbf{H} = [\mc F(X), \mc F(X)] / [\mc F(X), \mc R(X)]$ is a universal central extension of $H$.
\end{proposition}

We now refine this construction to more closely reflect our needs. Let $H$ be a perfect group and $\mc U = \set{U_1,\dots,U_m}$ be a set of subgroups which generate a perfect group $H$. Then let $\mc F(\mc U)$ be the free product of the groups $U_i$ and $\mc R(\mc U)$ be the kernel of the obvious map $\mc F(\mc U) \to H$ (ie, the relations on $H$ with generating groups $\mc U$).

\begin{corollary}
\label{cor:perf-central}
If $H$ and $\mc U$ are as described above, then $\widetilde{H} = [\mc F(\mc U),\mc F(\mc U)] / [\mc F(\mc U), \mc R(\mc U)]$ is a perfect central extension of $H$.
\end{corollary}

\begin{proof}
The corollary follows immediately when one notices that $\widetilde{H}$ is an intermediate factor of $\mbf{H} \to H$, and $\mbf{H}$ must be perfect. To see that $\mbf{H}$ is perfect, suppose that $\mbf H' = [\mbf H, \mbf H] \not= \mbf H$. Then $\mbf H'$ still covers $H$, since $H$ is perfect. Hence it is a central extension, and there is a homomorphism $\mbf{H} \to \mbf{H}' \subset \mbf H$ extending the identiy on $G$ by the universal property of $\mbf H$. Since $\id$ is also a homomorphism extending the identity on $H$ from $\mbf H$ to $\mbf H$, from the uniqueness property we conclude that $\mbf H' = \mbf H$.
\end{proof}

We will apply the above results to the group $G$, which is generated by its Lyapunov subgroups with respect to $\alpha$ (Lemma \ref{lem:generates}). Recall that $\mc P$ is the free product of such subgroups, and that $[\mc P, \mc P]$, the commutator subgroup, induces a sufficient Lyapunov subclass. We introduce the following notation: for the path and cycle spaces $\mc P$, $\mc C$ and $\mc C^c$, we let $\mc P_B$, $\mc C_B$ and $\mc C^c_B$ be their intersections with $[\mc P,\mc P]$ (the subscript $B$ is for ``bracket'').

\subsection{The Main Argument}

\begin{theorem}
\label{thm:key-technical}
If $\Phi : \mc P \to \R$ is continuous, $\Phi|_{\mc C}$ is a homomorphism whose kernel is normal in $\mc P$, and $\Phi(\mc S) = 0$, then $\Phi(\mc C_B) = 0$
\end{theorem}

\begin{proof}
We first show that $\Phi([\mc P, \mc C^c]) = 0$. It suffices to show that if $\rho \in \mc P$ is a path inside the $m\tth$ factor of $G$, and $\sigma \in  \mc C^c$ is also in the $m\tth$ factor of $G$, then $\Phi([\rho,\sigma]) = e$ (we may freely commute the others by Lemma \ref{lem:path-directprod} and since $\Phi(\mc S) = \set{e}$). Since $\Gamma$ is irreducible, we may choose $\gamma_i$ such that $\pi_m(\gamma_i)$ converges to $\pi_m(\omega(\rho))$. Choose a local section $s_m : U_m \to \mc P_m$ for $G_m$ (Lemma \ref{lem:loc-sec}), so that $\rho_i = \rho\cdot s_m(\omega(\rho)^{-1}\pi_m(\gamma_i))$ converges to $\rho$ in the group topology of $\mc P$. Note that $\omega(\rho_i) = \pi_m(\gamma_i)$, and $\rho_i$ is a path with components coming only from the $m\tth$ factor of $G$. Choose paths $\eta_i$ such that $\omega(\eta_i) = \prod_{k\not= m} \pi_k(\gamma_i)$. Then by construction $\omega(\rho_i \cdot \eta_i) = \gamma
_i$ and hence $\rho_i \cdot \eta_i \in \mc C$. But since distinct factors commute (Lemma \ref{lem:path-directprod}) and the target is abelian:

\[ \Phi([\rho,\sigma]) = \lim_{i \to \infty} \Phi([\rho_i,\sigma]) = \lim_{i \to \infty} \Phi([\rho_i \cdot \eta_i,\sigma]) = 0  \]

Let $Z = \mc C^c_B / \ker \Phi|_{\mc C^c_B}$ and $\widetilde{G} =  \mc P_B / \ker \Phi|_{\mc C^c_B}$. We may construct the topological extension:

\[ 1 \to Z \to \widetilde{G} \to G \to 1 \]

$Z$ must be a Lie group, since it is locally path connected and has a continuous injective homomorphism to a finite dimensional space. We claim that $Z$ is central in $\widetilde{G}$. Indeed, we have shown above that $\ker \Phi|_{\mc C^c_B}$  contains $[\mc P, \mc C^c]$. Observe that $\mc C^c = \ker (\mc P \to G)$. Then $\widetilde{G}$ will be a perfect central extension of $G$ by Corollary \ref{cor:perf-central}. But by \cite[Proposition 5.8, Lemma 5.9]{vinhage}, there are no nontrivial perfect Lie central extensions of simply connected, semisimple Lie groups. So $\Phi(\mc C^c_B) = 0$. But then $\Phi$ induces a homomorphism from $\mc C_B / \mc C^c_B = \Gamma$, which must vanish by our assumption on $\Gamma$ (it has no homomorphisms to $\R$ by assumption).

\end{proof}

\subsection{Proof of Theorem \ref{thm:main}}
\label{sec:twisted-vanishing}

We first address the untwisted case. By Proposition \ref{prop:pcf-suff}, it suffices to show that the periodic cycle functional vanishes on cycles in some suffcient subclass of Lyapunov paths. By Lemma \ref{lem:subgroup-sufficient}, it suffices to show that $P_\beta$ vanishes on the cycles in $[\mc P,\mc P] \subset \mc P$. We know that $P_\beta$ vanishes on $\mc S$ by Lemma \ref{lem:stable-vanishing}. Hence $P_\beta$ meets the criteria of Theorem \ref{thm:key-technical}, so $P_\beta(\mc C^c \cap [\mc P, \mc P]) = 0$. Hence it descends to a homomorphism for $\mc C_B / \mc C^c_B=  \Gamma$, and no such homomorphism exists by assumption. This is exactly the statement that $P_\beta$ vanishes on cycles in this sufficient Lyapunov subclass.

In the case of a twisted Weyl chamber flow, observe that we may follow the same techniques as the twisted case to get vanishing on $\mc C \cap [\mc P_{ss},\mc P_{ss}]$. Furthermore, notice that $[\mc P_{ss},\mc P_{ss}] \cdot \mc P_V$ corresponds to a sufficient subclass, and that 

\[ \mc C^c \cap ([\mc P_{ss},\mc P_{ss}] \cdot \mc P_V) \pmod {\mc S} = \mc C^c \cap [\mc P_{ss},\mc P_{ss}] \pmod {\mc S} \] 

So we get that $P_\beta$ descends to a homomorphism from $\Lambda$ with the same image. All such homomorphisms must vanish, see \cite{vinhage-wang}.

\bibliographystyle{plain}
\bibliography{local-rigidity-refs}

\end{document}